\documentclass{article}%
\usepackage{amsmath}
\usepackage{amsfonts}
\usepackage{amssymb}
\usepackage{graphicx}%
\newtheorem{theorem}{Theorem}

\newtheorem{corollary}[theorem]{Corollary}

\newtheorem{definition}[theorem]{Definition}

\newtheorem{lemma}[theorem]{Lemma}

\newtheorem{remark}[theorem]{Remark}

\newenvironment{proof}[1][Proof]{\noindent\textbf{#1.} }{\ \rule{0.5em}{0.5em}}
\numberwithin{equation}{section}
\numberwithin{theorem}{section}
\begin{document}

\title{The local sharp maximal function and BMO \\on locally homogeneous spaces}
\author{Marco Bramanti, Maria Stella Fanciullo\thanks{2000 AMS Classification: Primary
42B25. Secondary: 42B35, 46E30. Key-words: Locally homogeneous space; local
sharp maximal function; local BMO; Fefferman-Stein inequality; John-Nirenberg
inequality}}
\maketitle

\begin{abstract}
We prove a local version of Fefferman-Stein inequality for the local sharp
maximal function, and a local version of John-Nirenberg inequality for locally
BMO functions, in the framework of locally homogeneous spaces, in the sense of
Bramanti-Zhu \cite{BZ}.

\end{abstract}

\section{Introduction}

Real analysis and the theory of singular integrals have been developed first
in the Euclidean setting and then in more general contexts, in view of their
applications to harmonic analysis, partial differential equations, and complex
analysis. Around 1970 the theory of \emph{spaces of homogeneous type}, that is
quasi-metric doubling measure spaces,\emph{ }started to be systematically
developed in the monograph by Coifman-Weiss \cite{CW} and was successfully
applied to several fields. Much more recently, some problems arising from the
quest of a-priori estimates for PDEs suggested that real analysis would be a
more flexible and useful tool if its concepts and results were stated also in
a \emph{local }version. The meaning of this localization is, roughly speaking,
the following: we want an abstract theory which, when applied to the concrete
setting of a bounded domain $\Omega\subset\mathbb{R}^{n}$ endowed with a local
quasidistance $\rho$ and a locally doubling measure $d\mu$ (typically, the
Lebesgue measure), brings to integrals over metric balls $B_{r}\left(
x\right)  $ \emph{properly contained} in $\Omega$, and never requires to
compute integrals over sets of the kind $B_{r}\left(  x\right)  \cap\Omega,$
as happens when we apply the standard theory of spaces of homogeneous type to
a bounded domain $\Omega$. These versions, however, are not easily obtained
\emph{a posteriori }from the well established theory; instead, they require a
careful analysis which often poses nontrivial new problems.

For instance, global $L^{p}$ estimates for certain operators of
Ornstein-Uhlenbeck type were proved in \cite{BCLP} using results from a theory
of nondoubling spaces, developed in \cite{B}, which in particular applies to
certain locally doubling spaces. Bramanti-Zhu in \cite{BZ} developed instead a
theory of singular and fractional integrals in locally doubling spaces; these
results were applied in \cite{BZ2} to the proof of $L^{p}$ and Schauder
estimates for nonvariational operators structured on H\"{o}rmander's vector fields.

The aim of this paper is to continue the theory of locally doubling spaces, as
started in \cite{BZ}, with two main results: a local version of
Fefferman-Stein's theorem regarding the sharp maximal function, and a local
version of John-Nirenberg inequality about $BMO$ functions. The first of these
results, in its Euclidean version, is a key ingredient of a novel approach to
the proof of $L^{p}$ estimates for nonvariational elliptic and parabolic
operators with possibly discontinuous coefficients, first devised by Krylov in
\cite{K}; the present extension can open the way to the application of these
techniques to more general differential operators, as shown in \cite{BT}.

\bigskip

\noindent\textbf{Comparison with the existent literature and main results.}
The \textquotedblleft sharp maximal function\textquotedblright\ was introduced
in the Euclidean context by Fefferman-Stein in \cite{FS}, where the related
$L^{p}$ inequality was proved. In some spaces of generalized homogeneous type,
this operator has been introduced and studied by Lai in \cite{L}. \emph{Local}
sharp maximal functions have been studied by several authors (with different
definitions). Here we follow the approach of Iwaniec \cite{Iw} who proves, in
the Euclidean context, a version of local sharp maximal inequality. His proof
relies on a clever adaptation of Calder\'{o}n-Zygmund decomposition, and the
striking fact is that this construction can be adapted quite naturally to the
abstract context of locally homogeneous spaces, exploiting the properties of
the \textquotedblleft dyadic cubes\textquotedblright\ abstractly constructed
in this framework in \cite{BZ}. Our first main result is the sharp maximal
inequality stated in Theorem \ref{Thm FS}. This statement involves dyadic
cubes and the dyadic local sharp maximal function (see Definition
\ref{Def dyadic sharp}); since, however, these \textquotedblleft
cubes\textquotedblright\ are abstract objects which in the concrete
application of the theory are not easily visualized, it is convenient to
derive from Theorem \ref{Thm FS} some consequences formulated in the language
of metric balls and the local sharp maximal function (defined by means of
balls, instead of dyadic cubes, see Definition \ref{Def sharp}). These
results, more easily applicable, are Corollaries \ref{coroll 1},
\ref{coroll 2}, \ref{coroll 3}. 

The space $BMO$ of functions with bounded mean oscillation was introduced in
\cite{JN}, where the famous \textquotedblleft John-Nirenberg
inequality\textquotedblright\ is proved. Versions of this space and this inequality in spaces
of homogeneous type have been given by several authors, starting with Buckley
\cite{Bu} (see also Caruso-Fanciullo \cite{CF} and Dafni-Yue \cite{DY}). To adapt this result to our context,
we follow the approach contained in Mateu, Mattila, Nicolau, Orobitg
\cite[Appendix]{MMNO}, see also Castillo, Ramos Fern\'{a}ndez, Trousselot
\cite{CFT}. Our main result is Theorem \ref{john-nirenberg}, with its useful
consequence, Theorem \ref{Thm BMOp}, stating that we can equivalently compute
the mean oscillation of a function or its $L^{p}$ version for any $p\in\left(
1,\infty\right)  $, always computing averages over small balls.

\bigskip

\noindent\textbf{Plan of the paper. }Section 2 contains some basic facts about
locally doubling spaces; in Section 3 the local sharp maximal function is
studied and several $L^{p}$ inequalities are proved about it; in section 4 the
local John-Nirenberg inequality is proved.

\section{Preliminaries about locally homogeneous spaces}

We start recalling the abstract context of locally homogeneous spaces, as
introduced in \cite{BZ}.

(H1) Let $\Omega$ be a set, endowed with a function $\rho:\Omega\times
\Omega\rightarrow\lbrack0,\infty)$ such that for any $x,y\in\Omega$:

(a) $\rho\left(  x,y\right)  =0\Leftrightarrow x=y;$

(b) $\rho\left(  x,y\right)  =\rho\left(  y,x\right)  .$

For any $x\in\Omega,r>0,$ let us define the ball%
\[
B\left(  x,r\right)  =\left\{  y\in\Omega:\rho\left(  x,y\right)  <r\right\}
.
\]
These balls can be used to define a topology in $\Omega$, saying that
$A\subset\Omega$ is open if for any $x\in A$ there exists $r>0$ such that
$B\left(  x,r\right)  \subset A$. Also, we will say that $E\subset\Omega$ is
bounded if $E$ is contained in some ball.

Let us assume that:

(H2) (a) the balls are open with respect to this topology;

(H2) (b) for any $x\in\Omega$ and $r>0$ the closure of $B\left(  x,r\right)  $
is contained in $\left\{  y\in\Omega:\rho\left(  x,y\right)  \leq r\right\}
.$

It can be proved (see \cite[Prop. 2.4]{BZ}) that the validity of conditions
(H2) (a) and (b) is equivalent to the following:

(H2') $\rho\left(  x,y\right)  $ is a continuous function of $x$ for any fixed
$y\in\Omega.$

(H3) Let $\mu$ be a positive regular Borel measure in $\Omega.$

(H4) Assume there exists an increasing sequence $\left\{  \Omega_{n}\right\}
_{n=1}^{\infty}$ of bounded measurable subsets of $\Omega,$ such that:%

\begin{equation}%
{\displaystyle\bigcup_{n=1}^{\infty}}
\Omega_{n}=\Omega\label{Hp 0}%
\end{equation}
and such for, any $n=1,2,3,...$:

(i) the closure of $\Omega_{n}$ in $\Omega$ is compact;

(ii) there exists $\varepsilon_{n}>0$ such that%
\begin{equation}
\left\{  x\in\Omega:\rho\left(  x,y\right)  <2\varepsilon_{n}\text{ for some
}y\in\Omega_{n}\right\}  \subset\Omega_{n+1}; \label{Hp 1}%
\end{equation}
\qquad

We also assume that:

(H5) there exists $B_{n}\geq1$ such that for any $x,y,z\in\Omega_{n}$%
\begin{equation}
\rho\left(  x,y\right)  \leq B_{n}\left(  \rho\left(  x,z\right)  +\rho\left(
z,y\right)  \right)  ;\label{Hp 2}%
\end{equation}

(H6) there exists $C_{n}>1$ such that for any $x\in\Omega_{n},0<r\leq
\varepsilon_{n}$ we have%
\begin{equation}
\text{ }0<\mu\left(  B\left(  x,2r\right)  \right)  \leq C_{n}\mu\left(
B\left(  x,r\right)  \right)  <\infty. \label{Hp 3}%
\end{equation}
(Note that for $x\in\Omega_{n}$ and $r\leq\varepsilon_{n}$ we also have
$B\left(  x,2r\right)  \subset\Omega_{n+1}$).

\begin{definition}
\label{Def loc hom space}We will say that $\left(  \Omega,\left\{  \Omega
_{n}\right\}  _{n=1}^{\infty},\rho,\mu\right)  $ is a \emph{locally
homogeneous space }if assumptions (H1) to (H6) hold.
\end{definition}

\noindent\textbf{Dependence on the constants. }The numbers $\varepsilon
_{n},B_{n},C_{n}$ will be called \textquotedblleft the constants of
$\Omega_{n}$\textquotedblright. It is not restrictive to assume that
$B_{n},C_{n}$ are nondecreasing sequences, and $\varepsilon_{n}$ is a
nonincreasing sequence. Throughout the paper our estimates, for a fixed
$\Omega_{n},$ will often depend not only on the constants of $\Omega_{n},$ but
also (possibly) on the constants of $\Omega_{n+1},\Omega_{n+2},\Omega_{n+3}$.
We will briefly say that \textquotedblleft a constant depends on
$n$\textquotedblright\ to mean this type of dependence.

\bigskip

In the language of \cite{CW}, $\rho$ is a \emph{quasidistance} in each set
$\Omega_{n};$ we can also say that it is a \emph{local quasidistance in}
$\Omega.$ We stress that the two conditions appearing in (H2) are logically
independent each from the other, and they do not follow from (\ref{Hp 2}),
even when $\rho$ is a quasidistance\emph{ }in $\Omega,$ that is when
$B_{n}=B>1$ for all $n$. If, however, $\rho$ is a \emph{distance }in $\Omega$,
that is $B_{n}=1$ for all $n$, then (H2) is automatically fulfilled. The
continuity of $\rho$ also implies that (\ref{Hp 2}) still holds for
$x,y,z\in\overline{\Omega}_{n}.$ Also, note that $\mu\left(  \Omega
_{n}\right)  <\infty$ for every $n$, since $\overline{\Omega}_{n}$ is compact.

\bigskip

The basic concepts about Vitali covering lemma, the local maximal function and
its $L^{p}$ bound can be easily adapted to this context:

\begin{lemma}
[Vitali covering Lemma]\label{Lemma Vitali}Let $E$ be a measurable subset of
$\Omega_{n}$ and let $\left\{  B\left(  x_{\lambda},r_{\lambda}\right)
\right\}  _{\lambda\in\Lambda}$ be a family of balls with centers $x_{\lambda
}\in\Omega_{n}$ and radii $0<r_{\lambda}\leq r_{n}\equiv2\varepsilon
_{n}/\left(  2B_{n+1}+3B_{n+1}^{2}\right)  $, such that $E\subset%
{\displaystyle\bigcup\limits_{\lambda\in\Lambda}}
B\left(  x_{\lambda},r_{\lambda}\right)  $. Then one can select a countable
subcollection $\left\{  B\left(  x_{\lambda_{j}},r_{\lambda_{j}}\right)
\right\}  _{j=1}^{\infty}$ of mutually disjoint balls so that
\begin{equation}
E\subset%
{\displaystyle\bigcup\limits_{j=1}^{\infty}}
B\left(  x_{\lambda_{j}},K_{n}r_{\lambda_{j}}\right)  \label{Ch6-Vitali1}%
\end{equation}
with $K_{n}=\left(  2B_{n+1}+3B_{n+1}^{2}\right)  $ and, for some constant $c$
depending on $\Omega$,%
\begin{equation}
\underset{j=1}{\overset{\infty}{\sum}}\mu\left(  B\left(  x_{\lambda_{j}%
},r_{\lambda_{j}}\right)  \right)  \geq c\mu\left(  E\right)  .
\label{Ch6-Vitali2}%
\end{equation}

\end{lemma}

We can then give the following

\begin{definition}
Fix $\Omega_{n},\Omega_{n+1}$ and, for any $f\in L^{1}\left(  \Omega
_{n+1}\right)  $ define the \emph{local maximal function}%
\[
M_{\Omega_{n},\Omega_{n+1}}f\left(  x\right)  =\sup_{\substack{B\left(
\overline{x},r\right)  \ni x\\r\leq r_{n}}}\frac{1}{\mu\left(  B\left(
\overline{x},r\right)  \right)  }\int_{B\left(  \overline{x},r\right)
}\left\vert f\left(  y\right)  \right\vert d\mu\left(  y\right)  \text{ for
}x\in\Omega_{n}%
\]
where $r_{n}=2\varepsilon_{n}/\left(  2B_{n+1}+3B_{n+1}^{2}\right)  $ is the
same number appearing in Vitali Lemma. (Actually, the following theorem still
holds if this number $r_{n}$ is replaced by any smaller number).
\end{definition}

Then (see \cite[Thm. 8.3]{BZ}):

\begin{theorem}
\label{Thm maximal}Let $f$ be a measurable function defined on $\Omega_{n+1}.$
The following hold:

(a) If $f\in L^{p}\left(  \Omega_{n+1}\right)  $ for some $p\in\left[
1,\infty\right]  $, then $M_{\Omega_{n},\Omega_{n+1}}f$ is finite almost
everywhere in $\Omega_{n}$;

(b) if $f\in L^{1}\left(  \Omega_{n+1}\right)  $, then for every $t>0$,
\[
\mu\left(  \left\{  x\in\Omega_{n}:\left(  M_{\Omega_{n},\Omega_{n+1}%
}f\right)  \left(  x\right)  >t\right\}  \right)  \leq\frac{c_{n}}{t}%
\int_{\Omega_{n+1}}\left\vert f\left(  y\right)  \right\vert d\mu\left(
y\right)  ;
\]

(c) if $f\in L^{p}\left(  \Omega_{n+1}\right)  $, $1<p\leq\infty$, then
$M_{\Omega_{n},\Omega_{n+1}}f\in L^{p}\left(  \Omega_{n}\right)  $ and
\[
\left\Vert M_{\Omega_{n},\Omega_{n+1}}f\right\Vert _{L^{p}\left(  \Omega
_{n}\right)  }\leq c_{n,p}\left\Vert f\right\Vert _{L^{p}\left(  \Omega
_{n+1}\right)  }.
\]

\end{theorem}

By standard techniques, from the above theorem one can also prove the following:

\begin{theorem}
[Lebesgue differentiation theorem]\label{Thm Lebesgue}For every $f\in
L_{loc}^{1}\left(  \Omega_{n+1}\right)  $ and a.e. $x\in\Omega_{n}$ there
exists%
\[
\lim_{r\rightarrow0^{+}}\frac{1}{\mu\left(  B\left(  x,r\right)  \right)
}\int_{B\left(  x,r\right)  }f\left(  y\right)  d\mu\left(  y\right)
=f\left(  x\right)  .
\]
In particular, for every $f\in L_{loc}^{1}\left(  \Omega_{n+1}\right)  $ and
a.e. $x\in\Omega_{n},$%
\[
\left\vert f\left(  x\right)  \right\vert \leq M_{\Omega_{n},\Omega_{n+1}%
}f\left(  x\right)  \text{.}%
\]

\end{theorem}

\bigskip

A deep construction which is carried out in \cite[Thm. 8.3]{BZ}, adapting to
our local context an analogous construction developed in doubling spaces by
Christ \cite{Ch} is that of \emph{dyadic cubes}. Their relevant properties are
collected in the following:

\begin{theorem}
[Dyadic cubes]\label{Thm dyadic cubes}(See \cite[Thm. 3.1]{BZ}). Let $\left(
\Omega,\left\{  \Omega_{n}\right\}  _{n=1}^{\infty},\rho,\mu\right)  $ be a
locally homogeneous space. For any $n=1,2,3,...\ $there exists a collection
\[
\Delta_{n}=\left\{  Q_{\alpha}^{k}\subset\Omega,k=1,2,3...,\alpha\in
I_{k}\right\}
\]
(where $I_{k}$ is a set of indices) of open sets called \textquotedblleft
dyadic cubes subordinated to $\Omega_{n}$\textquotedblright\ , positive
constants $a_{0},c_{0},c_{1},c_{2},\delta\in\left(  0,1\right)  $ and a set
$E\subset\Omega_{n}$ of zero measure, such that for any $k=1,2,3...$ we have:

(a) $\forall\alpha\in I_{k},$ each $Q_{\alpha}^{k}$ contains a ball $B\left(
z_{\alpha}^{k},a_{0}\delta^{k}\right)  ;$

(b) $%
{\displaystyle\bigcup_{\alpha\in I_{k}}}
Q_{\alpha}^{k}\subset\Omega_{n+1};$

(c) $\forall\alpha\in I_{k},1\leq l\leq k,$ there exists $Q_{\beta}%
^{l}\supseteq Q_{\alpha}^{k};$

(d) $\forall\alpha\in I_{k},$ $\operatorname*{diam}\left(  Q_{\alpha}%
^{k}\right)  <c_{1}\delta^{k}$ and $\overline{Q_{\alpha}^{k}}\subset B\left(
z_{\alpha}^{k},c_{1}\delta^{k}\right)  \subset\Omega_{n+2};$

(e) $\ell\geq k\Longrightarrow\forall\alpha\in I_{k},\beta\in I_{l},Q_{\beta
}^{\ell}\subset Q_{\alpha}^{k}$ or $Q_{\beta}^{\ell}\cap Q_{\alpha}%
^{k}=\emptyset;$

(f) $\Omega_{n}\setminus%
{\displaystyle\bigcup_{\alpha\in I_{k}}}
Q_{\alpha}^{k}\subset E;$

(g) $\forall\alpha\in I_{k},$ $x\in Q_{\alpha}^{k}\setminus E,$ $j\geq1$ there
exists $Q_{\beta}^{j}\ni x;$

(h)
\begin{equation}
\mu\left(  B\left(  x,2r\right)  \cap Q_{\alpha}^{k}\right)  \leq c_{2}%
\mu\left(  B\left(  x,r\right)  \cap Q_{\alpha}^{k}\right)  \label{doubling}%
\end{equation}
for any $x\in Q_{\alpha}^{k}\setminus E,r>0.$ More precisely, for these $x$
and $r$ we have:%
\begin{equation}
\mu\left(  B\left(  x,r\right)  \cap Q_{\alpha}^{k}\right)  \geq\left\{
\begin{tabular}
[c]{ll}%
$c_{0}\mu\left(  B\left(  x,r\right)  \right)  $ & $\text{for }r\leq\delta
^{k}$\\
$c_{0}\mu\left(  Q_{\alpha}^{k}\right)  $ & $\text{for }r>\delta^{k}$%
\end{tabular}
\ \ \ \ \ \ \ \ \right.  \label{lower bounds cubes}%
\end{equation}

\end{theorem}

The sets $Q_{\alpha}^{k}$ can be thought as \emph{dyadic cubes of side length}
$\delta^{k}$. Note that $k$ is a \emph{positive }integer, so we are only
considering \emph{small }dyadic cubes. The cubes $\left\{  Q_{\alpha}%
^{k}\right\}  $ are \emph{subordinated to }a particular $\Omega_{n}$, meaning
that they essentially cover $\Omega_{n}$ (that is, their union covers
$\Omega_{n}$ up to a set of zero measure) and are contained in $\Omega_{n+1}$.
Note that the cubes $Q_{\alpha}^{k}$ and all the constants depend on $n$, so
we should write, more precisely%
\[
\left\{  Q_{\alpha}^{\left(  n\right)  ,k}\right\}  _{\alpha\in I_{k}^{\left(
n\right)  }};\delta_{\left(  n\right)  };a_{0,\left(  n\right)  },c_{0,\left(
n\right)  },c_{1,\left(  n\right)  },c_{2,\left(  n\right)  }\text{,}%
\]
but we will usually avoid this heavy notation.

In the proof of the above theorem, $\delta$ is chosen \emph{small enough}, so
it \ is not restrictive to assume $c_{1}\delta<2\varepsilon_{n+1}$, which
implies that the ball $B\left(  z_{\alpha}^{k},c_{1}\delta^{k}\right)  $
appearing in point (d) is $\subset\Omega_{n+2}$. (We remark this fact because
in \cite{BZ} the inclusion $B\left(  z_{\alpha}^{k},c_{1}\delta^{k}\right)
\subset\Omega_{n+2}$ is not stated).

Point (h) contains a crucial information: the triple $\left(  Q_{\alpha}%
^{k},\rho,d\mu\right)  $ is a space of homogeneous type in the sense of
Coifman-Weiss, that is the measure $\mu$ of $\rho$-balls \emph{restricted to}
$Q_{\alpha}^{k}$ is \emph{doubling}. Note that in our context this property
could fail to be true, instead, for the measure $\mu$ of $\rho$-balls
\emph{restricted to} \emph{a fixed }$\rho$\emph{-ball}.

We will also need the following:

\begin{lemma}
[Covering Lemma]\label{Covering Lemma}For every $n$ and every positive integer
$k$ large enough, the set $\Omega_{n}$ can be essentially covered by a finite
union of dyadic cubes $Q_{\alpha}^{k}$ (subordinated to $\Omega_{n+1}$) with
the following properties:

(i) $Q_{\alpha}^{k}\subset B\left(  z_{\alpha}^{k},c_{1}\delta^{k}\right)
\subset\Omega_{n+1}$

(ii) $B\left(  z_{\alpha}^{k},c_{1}\delta^{k}\right)  \subset F_{\alpha}^{k}$
(essentially), $F_{\alpha}^{k}\subset$ $B\left(  z_{\alpha}^{k},c^{\prime
}\delta^{k}\right)  \subset\Omega_{n+1}$, where the set $F_{\alpha}^{k}$ is a
finite union of dyadic cubes $Q_{\beta_{\alpha}}^{k}$, hence $F_{\alpha}^{k}$
is a space of homogeneous type, that is satisfies (h) of the previous theorem.
\end{lemma}

\begin{proof}
Since the whole $\Omega_{n+1}$ can be essentially covered by the union of the
dyadic cubes $Q_{\alpha}^{k}$ subordinated to $\Omega_{n+1}$, $\Omega_{n}$ is
essentially covered by a subfamily of these. By (d) in the previous theorem,
for each $Q_{\alpha}^{k}$ of these cubes, there exists a ball $B\left(
z_{\alpha}^{k},c_{1}\delta^{k}\right)  $ such that $Q_{\alpha}^{k}\subset
B\left(  z_{\alpha}^{k},c_{1}\delta^{k}\right)  \subset\Omega_{n+2}$. However,
since $Q_{\alpha}^{k},$ and then $B\left(  z_{\alpha}^{k},c_{1}\delta
^{k}\right)  $, contains a point of $\Omega_{n}$, for $k$ large enough
$B\left(  z_{\alpha}^{k},c_{1}\delta^{k}\right)  \subset\Omega_{n+1}$, that is
(i) holds.

Let $F_{\alpha}^{k}$ the union of all the dyadic cubes $Q_{\beta}^{k}$
intersecting $B\left(  z_{\alpha}^{k},c_{1}\delta^{k}\right)  $. Since
$B\left(  z_{\alpha}^{k},c_{1}\delta^{k}\right)  \subset\Omega_{n+1}$ which is
essentially covered by the union of all the dyadic cubes $Q_{\beta}^{k}$, then
$B\left(  z_{\alpha}^{k},c_{1}\delta^{k}\right)  $ is essentially covered by
$F_{\alpha}^{k}$. Since, by (d) in the previous theorem, $\operatorname*{diam}%
\left(  Q_{\beta}^{k}\right)  <c_{1}\delta^{k}$, and each $Q_{\beta}^{k}$
contained in $F_{\alpha}^{k}$ intersects $B\left(  z_{\alpha}^{k},c_{1}%
\delta^{k}\right)  $, $\operatorname*{diam}F_{\alpha}^{k}$ is comparable to
$\delta^{k}$, so $F_{\alpha}^{k}\subset$ $B\left(  z_{\alpha}^{k},c^{\prime
}\delta^{k}\right)  $, which is again contained in $\Omega_{n+1}$, for $k$
large enough, since the ball contains a point of $\Omega_{n}$. Finally, any
finite union of dyadic cubes satisfies the doubling condition, by
(\ref{doubling}) and \cite[Corollary 3.9]{BZ}.
\end{proof}

\section{The local sharp maximal function}

As we have explained in the Introduction, the proof of the sharp maximal
inequality will be achieved following the approach in \cite{Iw}, which
exploits a suitable version of Calder\'{o}n-Zygmund decomposition. We start
proving in the context of locally homogeneous spaces the following
decomposition lemma.

For any measurable set $E$ and function $f\in L^{1}\left(  E\right)  $, let
\[
f_{E}=\frac{1}{\left\vert E\right\vert }\int_{E}f.
\]

\begin{lemma}
\label{Lemma decomposition}For fixed $\Omega_{n},\Omega_{n+1}$ we consider the
family $\Delta_{n}$ of dyadic cubes built in Theorem \ref{Thm dyadic cubes}.
Let $Q_{\alpha_{1}}^{1}$ be a fixed dyadic cube (\textquotedblleft of first
generation\textquotedblright) and let $f\in L^{1}\left(  Q_{\alpha_{1}}%
^{1}\right)  $. For any $\lambda\geq a\equiv\left\vert f\right\vert
_{Q_{\alpha_{1}}^{1}}$ there exists a countable family $\mathcal{C}_{\lambda
}=\left\{  Q_{\lambda,j}\right\}  _{j=1,2,...}$ of pairwise disjoint dyadic
subcubes of $Q_{\alpha_{1}}^{1}$ such that:

(i) $\lambda<\left\vert f\right\vert _{Q_{\lambda,j}}\leq c_{n}\lambda$ for
$j=1,2,...$;

(ii) if $\lambda\geq\mu\geq a$ then each cube $Q_{\lambda,j}$ is a subcube of
one from the family $\mathcal{C}_{\mu}$;

(iii) $\left\vert f\left(  x\right)  \right\vert \leq\lambda$ for a.e. $x\in
Q_{\alpha_{1}}^{1}\setminus\bigcup\limits_{j}Q_{\lambda,j}$;

(iv) $\sum_{j}\left\vert Q_{\lambda,j}\right\vert \leq\left\vert \left\{  x\in
Q_{\alpha_{1}}^{1}:Mf\left(  x\right)  >\frac{\lambda}{c_{n}^{\prime}%
}\right\}  \right\vert $;

(v) $\left\vert \left\{  x\in Q_{\alpha_{1}}^{1}:Mf\left(  x\right)
>c_{n}^{\prime\prime}\lambda\right\}  \right\vert \leq c_{n}^{\prime
\prime\prime}\sum_{j}\left\vert Q_{\lambda,j}\right\vert $

\noindent where $c_{n},c_{n}^{\prime},c_{n}^{\prime\prime},c_{n}^{\prime
\prime\prime}$ are constants $>1$ only depending on $n$ and we let for
simplicity%
\[
Mf=M_{\Omega_{n+1},\Omega_{n+2}}\left(  f\chi_{Q_{\alpha_{1}}^{1}}\right)
\]
(i.e., the local maximal function is computed after extending $f$ to zero
outside $Q_{\alpha_{1}}^{1}$).
\end{lemma}

\begin{remark}
\label{remark Q1 small}As will be apparent from the proof, this lemma still
holds if instead of a fixed dyadic cube $Q_{\alpha_{1}}^{1}$ of the first
generation we fix a cube $Q_{\alpha_{k_{0}}}^{k_{0}}$ of some fixed generation
$k_{0}>1$. Throughout this section we will stick to the convention of
considering $Q_{\alpha_{1}}^{1}$ a cube of first generation, just to simplify
notation, however we must keep in mind that the results still hold under the
more general assumption on $Q_{\alpha_{k_{0}}}^{k_{0}}$.\ Or, saying this with
other words, we can think that the cube $Q_{\alpha_{1}}^{1}$ appearing in
Theorem \ref{Thm FS} and Corollaries \ref{coroll 1} and \ref{coroll 2} has
diameter as small as we want.
\end{remark}

\begin{proof}
By point (c) in Thm. \ref{Thm dyadic cubes} for every dyadic cube $Q\subset
Q_{\alpha_{1}}^{1}$ there exists an increasing chain of dyadic cubes%
\begin{equation}
Q=Q_{\alpha_{k}}^{k}\subset Q_{\alpha_{k-1}}^{k-1}\subset...\subset
Q_{\alpha_{1}}^{1}. \label{ascending chain}%
\end{equation}
For a fixed $\lambda\geq a=\left\vert f\right\vert _{Q_{\alpha_{1}}^{1}}$, in
order to define the family $\mathcal{C}_{\lambda}$ we say that $Q\in
\mathcal{C}_{\lambda}$ if, with the notation (\ref{ascending chain}),%
\[
\lambda<\left\vert f\right\vert _{Q}\text{ and }\left\vert f\right\vert
_{Q_{\alpha_{s}}^{s}}\leq\lambda\text{ for }s=1,2,...,k-1.
\]
Note that any two cubes in $\mathcal{C}_{\lambda}$ are disjoint, otherwise by
point (e) in Thm. \ref{Thm dyadic cubes} one should be contained in the other,
so they would be two different steps in the same chain (\ref{ascending chain}%
), which contradicts our rule of choice. Let us show that $\mathcal{C}%
_{\lambda}$ satisfies properties (i)-(iv).

(i). For $Q=Q_{\alpha_{k}}^{k}\in\mathcal{C}_{\lambda}$, by construction,
$\lambda<\left\vert f\right\vert _{Q}$ and $\left\vert f\right\vert
_{Q_{\alpha_{k-1}}^{k-1}}\leq\lambda$ hence%
\[
\left\vert f\right\vert _{Q}\leq\frac{1}{\left\vert Q_{\alpha_{k}}%
^{k}\right\vert }\int_{Q_{\alpha_{k-1}}^{k-1}}\left\vert f\right\vert
\leq\frac{\left\vert Q_{\alpha_{k-1}}^{k-1}\right\vert }{\left\vert
Q_{\alpha_{k}}^{k}\right\vert }\left\vert f\right\vert _{Q_{\alpha_{k-1}%
}^{k-1}}\leq c_{n}\lambda
\]
since by points (a) and (d) in Thm. \ref{Thm dyadic cubes},%
\[
B\left(  z_{\alpha}^{k},a_{0}\delta^{k}\right)  \subset Q_{\alpha_{k}}%
^{k}\subset Q_{\alpha_{k-1}}^{k-1}\subset B\left(  z_{\alpha_{k-1}}%
^{k-1},c_{1}\delta^{k-1}\right)  ,
\]
hence by the locally doubling condition%
\[
\frac{\left\vert Q_{\alpha_{k-1}}^{k-1}\right\vert }{\left\vert Q_{\alpha_{k}%
}^{k}\right\vert }\leq c_{n}%
\]
for some $c_{n}$ only depending on $n$ (in particular, independent of $k$).
Hence (i) is proved.

(ii). For $Q=Q_{\alpha_{k}}^{k}\in\mathcal{C}_{\lambda},\lambda\geq\mu\geq a$
we have%
\[
\left\vert f\right\vert _{Q_{\alpha_{k}}^{k}}>\lambda\geq\mu
\]
hence in the chain (\ref{ascending chain}) there is an $l$ such that
$\left\vert f\right\vert _{Q_{\alpha_{l}}^{l}}>\mu,\left\vert f\right\vert
_{Q_{\alpha_{l-1}}^{l-1}}\leq\mu$. This means that in the chain
(\ref{ascending chain}) there is a cube $Q^{\prime}=Q_{\alpha_{l}}^{l}%
\in\mathcal{C}_{\mu}$, and $Q^{\prime}\supset Q$.

(iii). Let $x\in Q_{\alpha_{1}}^{1}\setminus\bigcup\limits_{j}Q_{\lambda,j}$
and let $Q$ be any dyadic cube such that $x\in Q\subset Q_{\alpha_{1}}^{1}$.
Consider again the chain (\ref{ascending chain}) starting with $Q$. By our
choice of $x$, none of the cubes $Q_{\alpha_{l}}^{l}$ in this chain belongs to
$\mathcal{C}_{\lambda}$, and this means that $\left\vert f\right\vert
_{Q_{\alpha_{l}}^{l}}\leq\lambda$. Then by point (g) in Thm.
\ref{Thm dyadic cubes}, for a.e. $x\in Q_{\alpha_{1}}^{1}\setminus
\bigcup\limits_{j}Q_{\lambda,j}$ there exists a decreasing sequence of dyadic
cubes $\left\{  Q_{\alpha_{l}}^{l}\right\}  $ such that $\left\vert
f\right\vert _{Q_{\alpha_{l}}^{l}}\leq\lambda$ and $\cap Q_{\alpha_{l}}%
^{l}=\left\{  x\right\}  $. By Lebesgue's differentiation theorem, (iii) follows.

(iv). Let%
\[
f^{\ast}\left(  x\right)  =\sup_{x\in Q\in\Delta_{n}}\left\vert f\right\vert
_{Q}.
\]
In the previous point we have proved that for a.e. $x\in Q_{\alpha_{1}}%
^{1}\setminus\bigcup\limits_{j}Q_{\lambda,j}$ and dyadic cube $Q$ such that
$x\in Q\subset Q_{\alpha_{1}}^{1}$, we have $\left\vert f\right\vert _{Q}%
\leq\lambda$. Hence%
\[
f^{\ast}\left(  x\right)  \leq\lambda\text{ for a.e. }x\in Q_{\alpha_{1}}%
^{1}\setminus\bigcup\limits_{j}Q_{\lambda,j}.
\]
Conversely, if $x\in\bigcup\limits_{j}Q_{\lambda,j}$ then $\left\vert
f\right\vert _{Q_{\lambda,j}}>\lambda$ hence $f^{\ast}\left(  x\right)
>\lambda.$ These two facts mean that, up to a set of zero measure,%
\[
\bigcup\limits_{j}Q_{\lambda,j}=\left\{  x\in Q_{\alpha_{1}}^{1}:f^{\ast
}\left(  x\right)  >\lambda\right\}
\]
hence, since the $\left\{  Q_{\lambda,j}\right\}  _{j}$ are pairwise
disjoint,
\[
\sum_{j}\left\vert Q_{\lambda,j}\right\vert =\left\vert \left\{  x\in
Q_{\alpha_{1}}^{1}:f^{\ast}\left(  x\right)  >\lambda\right\}  \right\vert .
\]
However, again by points (a) and (d) in Thm. \ref{Thm dyadic cubes} and the
locally doubling condition,%
\[
f^{\ast}\left(  x\right)  \leq c_{n}M_{\Omega_{n+1},\Omega_{n+2}}\left(
f\chi_{Q_{\alpha_{1}}^{1}}\right)  \left(  x\right)  \equiv c_{n}Mf\left(
x\right)  ,
\]
hence (iv) follows.

(v). Let $Q_{\lambda,j}\in\mathcal{C}_{\lambda}$. For some $k=2,3,...,$ we
will have $Q_{\lambda,j}=Q_{\alpha_{k}}^{k}$ and by points (a) and (d) in Thm.
\ref{Thm dyadic cubes},%
\[
B\left(  z_{\alpha}^{k},a_{0}\delta^{k}\right)  \subset Q_{\alpha_{k}}%
^{k}\subset B\left(  z_{\alpha_{k}}^{k},c_{1}\delta^{k}\right)
\]
for some $z_{\alpha}^{k}$. For a $K>1$ to be chosen later, let $KQ_{\lambda
,j}=B\left(  z_{\alpha_{k}}^{k},Kc_{1}\delta^{k}\right)  $. For any
$x\notin\bigcup\limits_{j}KQ_{\lambda,j}$ and any ball $B=B_{r}\left(
\overline{x}\right)  $ such that $x\in B_{r}\left(  \overline{x}\right)  $ and
$r\leq r_{0}$ we have, extending $f$ to zero outside $Q_{\alpha_{1}}^{1}$ if
$B\varsubsetneq Q_{\alpha_{1}}^{1}$,%
\[
\int_{B}\left\vert f\right\vert =\int_{B\setminus\bigcup\limits_{j}%
Q_{\lambda,j}}\left\vert f\right\vert +\sum_{j}\int_{B\cap Q_{\lambda,j}%
}\left\vert f\right\vert
\]
by point (iii)%
\[
\leq\lambda\left\vert B\right\vert +\sum_{j:B\cap Q_{\lambda,j}\neq\emptyset
}\int_{Q_{\lambda,j}}\left\vert f\right\vert
\]
by point (i)
\begin{equation}
\leq\lambda\left\vert B\right\vert +c_{n}\lambda\sum_{j:B\cap Q_{\lambda
,j}\neq\emptyset}\left\vert Q_{\lambda,j}\right\vert . \label{proof (v)}%
\end{equation}
Next, we need the following\medskip

\noindent\textbf{Claim.} There exists $K,H>1$ (only depending on $n$) such
that if $x\notin\bigcup\limits_{i}KQ_{\lambda,i}$ and $B_{r}\left(
\overline{x}\right)  \cap Q_{\lambda,j}\neq\emptyset$ then $Q_{\lambda
,j}\subset B_{Hr}\left(  \overline{x}\right)  $.\medskip

\noindent\textbf{Proof of the Claim.} Recall that
\begin{align*}
B\left(  z_{\alpha}^{k},a_{0}\delta^{k}\right)   &  \subset Q_{\lambda
,j}\subset B\left(  z_{\alpha_{k}}^{k},c_{1}\delta^{k}\right) \\
KQ_{\lambda,j}  &  =B\left(  z_{\alpha_{k}}^{k},Kc_{1}\delta^{k}\right)  .
\end{align*}
Since $B_{r}\left(  \overline{x}\right)  \cap Q_{\lambda,j}\neq\emptyset$, in
particular $B_{r}\left(  \overline{x}\right)  \cap B\left(  z_{\alpha}%
^{k},a_{0}\delta^{k}\right)  \neq\emptyset$ hence%
\[
\rho\left(  \overline{x},z_{\alpha}^{k}\right)  \leq B_{n+1}\left(
r+a_{0}\delta^{k}\right)  .
\]
Since $x\in B_{r}\left(  \overline{x}\right)  $,%
\[
\rho\left(  x,z_{\alpha}^{k}\right)  \leq B_{n+2}\left(  r+\rho\left(
\overline{x},z_{\alpha}^{k}\right)  \right)  \leq B_{n+2}\left(
r+B_{n+1}\left(  r+a_{0}\delta^{k}\right)  \right)
\]
and since $x\notin KQ_{\lambda,j},$%
\[
B_{n+2}\left(  r+B_{n+1}\left(  r+a_{0}\delta^{k}\right)  \right)
>Kc_{1}\delta^{k}%
\]
which, picking $K=\frac{2B_{n+1}B_{n+2}a_{0}}{c_{1}}$, gives%
\[
\delta^{k}<\frac{\left(  1+B_{n+1}\right)  }{B_{n+1}a_{0}}r.
\]
Then $Q_{\lambda,j}\subset B\left(  z_{\alpha_{k}}^{k},c_{1}\delta^{k}\right)
\subset B_{Hr}\left(  \overline{x}\right)  $ for a suitable $H$ depending on
$n$, namely for $z\in Q_{\lambda,j},$%
\begin{align*}
\rho\left(  z,\overline{x}\right)   &  \leq B_{n+1}\left(  c_{1}\delta
^{k}+\rho\left(  \overline{x},z_{\alpha}^{k}\right)  \right)  \leq
B_{n+1}\left(  c_{1}\delta^{k}+B_{n+1}\left(  r+a_{0}\delta^{k}\right)
\right) \\
&  \leq\delta^{k}\left(  B_{n+1}c_{1}+B_{n+1}^{2}a_{0}\right)  +B_{n+1}^{2}r\\
&  \leq\frac{\left(  1+B_{n+1}\right)  }{B_{n+1}a_{0}}r\left(  B_{n+1}%
c_{1}+B_{n+1}^{2}a_{0}\right)  +B_{n+1}^{2}r\equiv Hr,
\end{align*}
which proves the Claim.

Let us come back to the proof of (v). By (\ref{proof (v)}) and the Claim we
have (since the dyadic cubes in the sum are disjoint)%
\[
\int_{B_{r}\left(  \overline{x}\right)  }\left\vert f\right\vert \leq
\lambda\left\vert B_{r}\left(  \overline{x}\right)  \right\vert +c_{n}%
\lambda\left\vert B_{Hr}\left(  \overline{x}\right)  \right\vert
\]
and, by the locally doubling condition, for $r\leq r_{n}$ small enough,%
\[
\left\vert f\right\vert _{B}\leq c_{n}^{\prime}\lambda
\]
for every $B\ni x\in Q_{\alpha_{1}}^{1}\setminus\bigcup\limits_{j}%
KQ_{\lambda,j}$, that is%
\[
Mf\left(  x\right)  \leq c_{n}^{\prime}\lambda
\]
for any such $x$, so that%
\[
\left\{  x\in Q_{\alpha_{1}}^{1}:Mf\left(  x\right)  >c_{n}^{\prime}%
\lambda\right\}  \subset\bigcup\limits_{j}KQ_{\lambda,j}%
\]
and%
\[
\left\vert \left\{  x\in Q_{\alpha_{1}}^{1}:Mf\left(  x\right)  >c_{n}%
^{\prime}\lambda\right\}  \right\vert \leq\sum_{j}\left\vert KQ_{\lambda
,j}\right\vert \leq c_{n}^{\prime\prime}\sum_{j}\left\vert Q_{\lambda
,j}\right\vert .
\]

\end{proof}

We can now prove the following local analog of Fefferman-Stein inequality. Let
us first define a version of \emph{dyadic sharp maximal function}:

\begin{definition}
\label{Def dyadic sharp}For $f\in L^{1}\left(  Q_{\alpha_{1}}^{1}\right)  $,
$x\in Q_{\alpha_{1}}^{1}$, let%
\[
f_{\Delta}^{\#}\left(  x\right)  =\sup_{\substack{x\ni Q\in\Delta
_{n}\\Q\subset Q_{\alpha_{1}}^{1}}}\frac{1}{\left\vert Q\right\vert }\int
_{Q}\left\vert f-f_{Q}\right\vert .
\]

\end{definition}

Note that this definition involves only the values of $f$ in $Q_{\alpha_{1}%
}^{1}$ (there is no need of extending $f$ outside that cube).

\begin{theorem}
[Local Fefferman-Stein inequality]\label{Thm FS}Let $f\in L^{1}\left(
Q_{\alpha_{1}}^{1}\right)  $ and assume $f_{\Delta}^{\#}\in L^{p}\left(
Q_{\alpha_{1}}^{1}\right)  $ for some $p\in\lbrack1,+\infty).$ Then $Mf\in
L^{p}\left(  Q_{\alpha_{1}}^{1}\right)  $ and%
\begin{equation}
\left(  \frac{1}{\left\vert Q_{\alpha_{1}}^{1}\right\vert }\int_{Q_{\alpha
_{1}}^{1}}\left(  Mf\right)  ^{p}\right)  ^{1/p}\leq c_{n,p}\left\{  \left(
\frac{1}{\left\vert Q_{\alpha_{1}}^{1}\right\vert }\int_{Q_{\alpha_{1}}^{1}%
}\left(  f_{\Delta}^{\#}\right)  ^{p}\right)  ^{1/p}+\left(  \frac
{1}{\left\vert Q_{\alpha_{1}}^{1}\right\vert }\int_{Q_{\alpha_{1}}^{1}%
}\left\vert f\right\vert \right)  \right\}  \label{FS ineq 1}%
\end{equation}
for some constant $c_{n,p}$ only depending on $n,p,$ where, as above,%
\[
Mf=M_{\Omega_{n+1},\Omega_{n+2}}\left(  f\chi_{Q_{\alpha_{1}}^{1}}\right)  .
\]

\end{theorem}

\begin{proof}
Let $a=\left\vert f\right\vert _{Q_{\alpha_{1}}^{1}}$. We start proving the
following estimate: for every $\lambda\geq2c_{n}a$ and every $A>0$,
\begin{equation}
\sum_{j}\left\vert Q_{\lambda,j}\right\vert \leq\left\vert \left\{  x\in
Q_{\alpha_{1}}^{1}:f_{\Delta}^{\#}\left(  x\right)  >\frac{\lambda}%
{A}\right\}  \right\vert +\frac{2}{A}\sum_{Q\in\mathcal{C}_{\lambda/2c_{n}}%
}\left\vert Q\right\vert . \label{dim_sharp3}%
\end{equation}

Fix $\lambda\geq2c_{n}a$. By point (i) in Lemma \ref{Lemma decomposition}, for
any $Q_{\lambda,k}\in\mathcal{C}_{\lambda},$
\[
\lambda<\left\vert f\right\vert _{Q_{\lambda,j}}\leq c_{n}\lambda;
\]
also, for any $Q\in\mathcal{C}_{\lambda/2c_{n}}$,
\[
\left\vert f\right\vert _{Q}\leq c_{n}\frac{\lambda}{2c_{n}}=\frac{\lambda}{2}%
\]
so that, for such $Q_{\lambda,j}$ and $Q$,%
\begin{align}
\frac{1}{\left\vert Q_{\lambda,j}\right\vert }\int_{Q_{\lambda,j}}\left\vert
f-f_{Q}\right\vert  &  \geq\left\vert f\right\vert _{Q_{\lambda,j}}-\left\vert
f\right\vert _{Q}>\lambda-\frac{\lambda}{2}=\frac{\lambda}{2}\nonumber\\
\left\vert Q_{\lambda,j}\right\vert  &  <\frac{2}{\lambda}\int_{Q_{\lambda,j}%
}\left\vert f-f_{Q}\right\vert . \label{dim_sharp1}%
\end{align}
By point (ii) in Lemma \ref{Lemma decomposition}, since $\lambda
>\lambda/2c_{n}$, any $Q_{\lambda,j}\in\mathcal{C}_{\lambda}$ is contained in
some $Q\in\mathcal{C}_{\lambda/2c_{n}};$ also, the cubes $Q\in\mathcal{C}%
_{\lambda/2c_{n}}$, like the cubes $Q_{\lambda,j}\in\mathcal{C}_{\lambda}$ are
pairwise disjoint, hence we can write%
\begin{equation}
\sum_{j}\left\vert Q_{\lambda,j}\right\vert =\sum_{Q\in\mathcal{C}%
_{\lambda/2c_{n}}}\sum_{\substack{Q_{\lambda,j}\in\mathcal{C}_{\lambda
}\\Q_{\lambda,j}\subset Q}}\left\vert Q_{\lambda,j}\right\vert .
\label{dim_sharp2}%
\end{equation}
For any $Q\in\mathcal{C}_{\lambda/2c_{n}},$ by (\ref{dim_sharp1}) and since
the $Q_{\lambda,j}$ are disjoint
\[
\sum_{\substack{Q_{\lambda,j}\in\mathcal{C}_{\lambda}\\Q_{\lambda,j}\subset
Q}}\left\vert Q_{\lambda,j}\right\vert \leq\sum_{\substack{Q_{\lambda,j}%
\in\mathcal{C}_{\lambda}\\Q_{\lambda,j}\subset Q}}\frac{2}{\lambda}%
\int_{Q_{\lambda,j}}\left\vert f-f_{Q}\right\vert \leq\frac{2}{\lambda}%
\int_{Q}\left\vert f-f_{Q}\right\vert .
\]
Let now fix a number $A>0$ and distinguish two cases:

a. If $\frac{1}{\left\vert Q\right\vert }\int_{Q}\left\vert f-f_{Q}\right\vert
\leq\frac{\lambda}{A},$ then
\[
\sum_{\substack{Q_{\lambda,j}\in\mathcal{C}_{\lambda}\\Q_{\lambda,j}\subset
Q}}\left\vert Q_{\lambda,j}\right\vert \leq\frac{2}{\lambda}\frac{\lambda}%
{A}\left\vert Q\right\vert =\frac{2}{A}\left\vert Q\right\vert .
\]

b. If $\frac{1}{\left\vert Q\right\vert }\int_{Q}\left\vert f-f_{Q}\right\vert
>\frac{\lambda}{A}$, then for every $x\in Q$%
\[
f_{\Delta}^{\#}\left(  x\right)  >\frac{\lambda}{A}\text{,}%
\]
that is%
\[
Q\subset\left\{  x\in Q_{\alpha_{1}}^{1}:f_{\Delta}^{\#}\left(  x\right)
>\frac{\lambda}{A}\right\}
\]
and%
\[
\sum_{\substack{Q_{\lambda,j}\in\mathcal{C}_{\lambda}\\Q_{\lambda,j}\subset
Q}}\left\vert Q_{\lambda,j}\right\vert \leq\left\vert Q\cap\left\{  x\in
Q_{\alpha_{1}}^{1}:f_{\Delta}^{\#}\left(  x\right)  >\frac{\lambda}%
{A}\right\}  \right\vert .
\]
In any case we can write%
\[
\sum_{\substack{Q_{\lambda,j}\in\mathcal{C}_{\lambda}\\Q_{\lambda,j}\subset
Q}}\left\vert Q_{\lambda,j}\right\vert \leq\left\vert Q\cap\left\{  x\in
Q_{\alpha_{1}}^{1}:f_{\Delta}^{\#}\left(  x\right)  >\frac{\lambda}%
{A}\right\}  \right\vert +\frac{2}{A}\left\vert Q\right\vert .
\]
Adding up these inequalities for $Q\in\mathcal{C}_{\lambda/2c_{n}}$, recalling
(\ref{dim_sharp2}) and the fact that the cubes $Q$ are pairwise disjoint, we
get (\ref{dim_sharp3}).

By (\ref{dim_sharp3}) and points (v), (iv) in Lemma \ref{Lemma decomposition},
we have
\begin{align}
&  \left\vert \left\{  x\in Q_{\alpha_{1}}^{1}:Mf\left(  x\right)
>c_{n}^{\prime\prime}\lambda\right\}  \right\vert \nonumber\\
&  \leq c_{n}^{\prime\prime\prime}\left(  \left\vert \left\{  x\in
Q_{\alpha_{1}}^{1}:f_{\Delta}^{\#}\left(  x\right)  >\frac{\lambda}%
{A}\right\}  \right\vert +\frac{2}{A}\sum_{Q\in\mathcal{C}_{\lambda/2c_{n}}%
}\left\vert Q\right\vert \right) \nonumber\\
&  \leq c_{n}^{\prime\prime\prime}\left(  \left\vert \left\{  x\in
Q_{\alpha_{1}}^{1}:f_{\Delta}^{\#}\left(  x\right)  >\frac{\lambda}%
{A}\right\}  \right\vert +\frac{2}{A}\left\vert \left\{  x\in Q_{\alpha_{1}%
}^{1}:Mf\left(  x\right)  >\frac{\lambda}{2c_{n}c_{n}^{\prime}}\right\}
\right\vert \right)  \label{dim_sharp4}%
\end{align}
for any $A>0,\lambda\geq2c_{n}a$.

We now want to compute integrals using the identity (for any $F\in
L^{p}\left(  Q_{\alpha_{1}}^{1}\right)  ,1\leq p<\infty$)%
\[
\int_{Q_{\alpha_{1}}^{1}}\left\vert F\left(  y\right)  \right\vert ^{p}%
dy=\int_{0}^{+\infty}pt^{p-1}\left\vert \left\{  x\in Q_{\alpha_{1}}%
^{1}:\left\vert F\left(  x\right)  \right\vert >t\right\}  \right\vert dt.
\]
Letting%
\[
\mu\left(  t\right)  =\left\vert \left\{  x\in Q_{\alpha_{1}}^{1}:\left\vert
Mf\left(  x\right)  \right\vert >t\right\}  \right\vert
\]
and integrating (\ref{dim_sharp4}) for $\lambda\in\left(  2c_{n}a,N\right)  $
and any fixed $N>2c_{n}\left\vert f\right\vert _{Q_{\alpha_{1}}^{1}}$ after
multiplying by $p\lambda^{p-1}$ we have%
\begin{align}
&  \int_{2c_{n}a}^{N}p\lambda^{p-1}\mu\left(  c_{n}^{\prime\prime}%
\lambda\right)  d\lambda\nonumber\\
&  \leq c_{n}^{\prime\prime\prime}\left(  \int_{2c_{n}a}^{N}p\lambda
^{p-1}\left\vert \left\{  x\in Q_{\alpha_{1}}^{1}:f_{\Delta}^{\#}\left(
x\right)  >\frac{\lambda}{A}\right\}  \right\vert d\lambda+\frac{2}{A}%
\int_{2c_{n}a}^{N}p\lambda^{p-1}\mu\left(  \frac{\lambda}{2c_{n}c_{n}^{\prime
}}\right)  d\lambda\right)  \label{dim_sharp5}%
\end{align}

Changing variable in each of the three integrals in (\ref{dim_sharp5}) we get:%
\begin{align}
&  \int_{2c_{n}c_{n}^{\prime\prime}a}^{c_{n}^{\prime\prime}N}pt^{p-1}%
\mu\left(  t\right)  dt\nonumber\\
&  \leq\left(  c_{n}^{\prime\prime}\right)  ^{p}c_{n}^{\prime\prime\prime
}\left(  A^{p}\int_{0}^{+\infty}pt^{p-1}\left\vert \left\{  x\in Q_{\alpha
_{1}}^{1}:f_{\Delta}^{\#}\left(  x\right)  >t\right\}  \right\vert dt+\right.
\label{dim_sharp6}\\
&  \left.  \frac{2}{A}\left(  2c_{n}c_{n}^{\prime}\right)  ^{p}\int_{0}%
^{\frac{N}{2c_{n}c_{n}^{\prime}}}pt^{p-1}\mu\left(  t\right)  dt\right)
.\nonumber
\end{align}
Using also the elementary inequality%
\[
\int_{0}^{2c_{n}c_{n}^{\prime\prime}a}pt^{p-1}\mu\left(  t\right)  dt\leq
\int_{0}^{2c_{n}c_{n}^{\prime\prime}a}pt^{p-1}\left\vert Q_{\alpha_{1}}%
^{1}\right\vert dt=\left\vert Q_{\alpha_{1}}^{1}\right\vert \left(
2c_{n}c_{n}^{\prime\prime}a\right)  ^{p},
\]
together with (\ref{dim_sharp6}), since $\frac{N}{2c_{n}c_{n}^{\prime}%
}<N<c_{n}^{\prime\prime}N$ we get%
\begin{align*}
&  \int_{0}^{c_{n}^{\prime\prime}N}pt^{p-1}\mu\left(  t\right)  dt\\
&  \leq\left(  c_{n}^{\prime\prime}\right)  ^{p}c_{n}^{\prime\prime\prime
}\left(  A^{p}\int_{0}^{+\infty}pt^{p-1}\left\vert \left\{  x\in Q_{\alpha
_{1}}^{1}:f_{\Delta}^{\#}\left(  x\right)  >t\right\}  \right\vert dt\right.
\\
&  \left.  +\frac{2}{A}\left(  2c_{n}c_{n}^{\prime}\right)  ^{p}\int
_{0}^{c_{n}^{\prime\prime}N}pt^{p-1}\mu\left(  t\right)  dt\right)
+\left\vert Q_{\alpha_{1}}^{1}\right\vert \left(  2c_{n}c_{n}^{\prime\prime
}a\right)  ^{p}.
\end{align*}
Letting finally $A=4\left(  2c_{n}c_{n}^{\prime}c_{n}^{\prime\prime}\right)
^{p}c_{n}^{\prime\prime\prime}$ we conclude%
\begin{align*}
&  \int_{0}^{c_{n}^{\prime\prime}N}pt^{p-1}\mu\left(  t\right)  dt\\
&  \leq c_{n,p}\left(  \int_{0}^{+\infty}pt^{p-1}\left\vert \left\{  x\in
Q_{\alpha_{1}}^{1}:f_{\Delta}^{\#}\left(  x\right)  >t\right\}  \right\vert
dt+\left\vert Q_{\alpha_{1}}^{1}\right\vert \left\vert f\right\vert
_{Q_{\alpha_{1}}^{1}}^{p}\right)
\end{align*}
which implies that $Mf\in L^{p}\left(  Q_{\alpha_{1}}^{1}\right)  $ and%
\[
\left\Vert Mf\right\Vert _{L^{p}\left(  Q_{\alpha_{1}}^{1}\right)  }^{p}\leq
c_{n,p}\left(  \left\Vert f_{\Delta}^{\#}\right\Vert _{L^{p}\left(
Q_{\alpha_{1}}^{1}\right)  }^{p}+\left\vert Q_{\alpha_{1}}^{1}\right\vert
\left\vert f\right\vert _{Q_{\alpha_{1}}^{1}}^{p}\right)
\]
that is (\ref{FS ineq 1}).
\end{proof}

We now want to reformulate the above theorem in terms of balls, instead of
dyadic cubes, to make it more easily applicable to concrete situations. This
reformulation can be done in several ways. First of all, we introduce the
local sharp maximal functions defined by balls instead of cubes.

\begin{definition}
\label{Def sharp}For $f\in L_{loc}^{1}\left(  \Omega_{n+1}\right)  $,
$x\in\Omega_{n}$, let%
\[
f_{\Omega_{n},\Omega_{n+1}}^{\#}\left(  x\right)  =\sup_{\substack{B\left(
\overline{x},r\right)  \ni x\\\overline{x}\in\Omega_{n},r\leq\varepsilon_{n}%
}}\frac{1}{\left\vert B\left(  \overline{x},r\right)  \right\vert }%
\int_{B\left(  \overline{x},r\right)  }\left\vert f-f_{B\left(  \overline
{x},r\right)  }\right\vert .
\]

\end{definition}

Let us compare this function with its dyadic version $f_{\Delta}^{\#}$:

\begin{lemma}
\label{Lemma compare sharps}With the above notation, for any $x\in
Q_{\alpha_{1}}^{1},$%
\[
f_{\Delta}^{\#}\left(  x\right)  \leq c_{n}f_{\Omega_{n+1},\Omega_{n+2}}%
^{\#}\left(  x\right)
\]
for some constant $c_{n}$ only depending on $n$. Here the function $f$ can be
assumed either in $L_{loc}^{1}\left(  \Omega_{n+2}\right)  $ or in
$L^{1}\left(  Q_{\alpha_{1}}^{1}\right)  $ and extended to zero outside
$Q_{\alpha_{1}}^{1}$.
\end{lemma}

\begin{proof}
For any dyadic cube $Q=Q_{\alpha_{k}}^{k}\subset Q_{\alpha_{1}}^{1}%
\subset\Omega_{n+1}$ we have (see points (a), (d) in Theorem
\ref{Thm dyadic cubes})%
\[
B\left(  z_{\alpha}^{k},a_{0}\delta^{k}\right)  \subset Q\subset B\left(
z_{\alpha_{k}}^{k},c_{1}\delta^{k}\right)  \subset\Omega_{n+2}.
\]
Let us briefly write $B_{1},B_{2}$ in place of $B\left(  z_{\alpha}^{k}%
,a_{0}\delta^{k}\right)  ,B\left(  z_{\alpha_{k}}^{k},c_{1}\delta^{k}\right)
$. Then by the locally doubling condition
\[
\frac{\left\vert B_{2}\right\vert }{\left\vert Q\right\vert }\leq
\frac{\left\vert B_{2}\right\vert }{\left\vert B_{1}\right\vert }\leq c_{n}%
\]
and we can write%
\begin{align*}
\frac{1}{\left\vert Q\right\vert }\int_{Q}\left\vert f-f_{Q}\right\vert  &
\leq\frac{1}{\left\vert Q\right\vert }\int_{B_{2}}\left\vert f-f_{Q}%
\right\vert \\
&  \leq c_{n}\frac{1}{\left\vert B_{2}\right\vert }\int_{B_{2}}\left\vert
f-f_{Q}\right\vert \\
&  \leq c_{n}\left(  \frac{1}{\left\vert B_{2}\right\vert }\int_{B_{2}%
}\left\vert f-f_{B_{2}}\right\vert +\left\vert f_{Q}-f_{B_{2}}\right\vert
\right)  .
\end{align*}
Also,%
\begin{align*}
\left\vert f_{Q}-f_{B_{2}}\right\vert  &  =\left\vert \frac{1}{\left\vert
Q\right\vert }\int_{Q}\left(  f-f_{B_{2}}\right)  \right\vert \leq\frac
{1}{\left\vert Q\right\vert }\int_{Q}\left\vert f-f_{B_{2}}\right\vert \\
&  \leq c_{n}\frac{1}{\left\vert B_{2}\right\vert }\int_{B_{2}}\left\vert
f-f_{B_{2}}\right\vert
\end{align*}
hence%
\[
\frac{1}{\left\vert Q\right\vert }\int_{Q}\left\vert f-f_{Q}\right\vert \leq
c_{n}\left(  1+c_{n}\right)  \frac{1}{\left\vert B_{2}\right\vert }\int
_{B_{2}}\left\vert f-f_{B_{2}}\right\vert
\]
\bigskip and the assertion follows.
\end{proof}

Exploiting the previous Lemma and Theorem \ref{Thm Lebesgue} we can now
rewrite the statement of Theorem \ref{Thm FS} as follows:

\begin{corollary}
\label{coroll 1}Let $f\in L^{1}\left(  Q_{\alpha_{1}}^{1}\right)  $ and assume
$f_{\Omega_{n+1},\Omega_{n+2}}^{\#}\in L^{p}\left(  Q_{\alpha_{1}}^{1}\right)
$ for some $p\in\lbrack1,+\infty).$ Then $f\in L^{p}\left(  Q_{\alpha_{1}}%
^{1}\right)  $ and%
\begin{align*}
&  \left(  \frac{1}{\left\vert Q_{\alpha_{1}}^{1}\right\vert }\int
_{Q_{\alpha_{1}}^{1}}\left\vert Mf\right\vert ^{p}\right)  ^{1/p}\\
&  \leq c_{n,p}\left\{  \left(  \frac{1}{\left\vert Q_{\alpha_{1}}%
^{1}\right\vert }\int_{Q_{\alpha_{1}}^{1}}\left(  f_{\Omega_{n+1},\Omega
_{n+2}}^{\#}\right)  ^{p}\right)  ^{1/p}+\left(  \frac{1}{\left\vert
Q_{\alpha_{1}}^{1}\right\vert }\int_{Q_{\alpha_{1}}^{1}}\left\vert
f\right\vert \right)  \right\}
\end{align*}
for some constant $c_{n,p}$ only depending on $n,p$. Here $M$ is defined as in
Theorem \ref{Thm FS} and, again, the function $f$ can be assumed either in
$L_{loc}^{1}\left(  \Omega_{n+2}\right)  $ or in $L^{1}\left(  Q_{\alpha_{1}%
}^{1}\right)  $ and extended to zero outside $Q_{\alpha_{1}}^{1}$.
\end{corollary}

The following is also useful:

\begin{corollary}
\label{coroll 2}Let $B_{1}\subset Q_{\alpha_{1}}^{1}\subset B_{2}$ with
$B_{1},B_{2}$ concentric balls of comparable radii (like in the proof of Lemma
\ref{Lemma compare sharps}) and assume that $f\in L^{1}\left(  B_{2}\right)  $
with $f_{\Omega_{n+1},\Omega_{n+2}}^{\#}\in L^{p}\left(  Q_{\alpha_{1}}%
^{1}\right)  $ for some $p\in\lbrack1,+\infty)$ and $f_{B_{2}}=0$ (where the
function $f$ can be assumed either in $L_{loc}^{1}\left(  \Omega_{n+2}\right)
$ or in $L^{1}\left(  B_{2}\right)  $ and extended to zero outside $B_{2}$).
Then $f\in L^{p}\left(  Q_{\alpha_{1}}^{1}\right)  $ and we have%
\begin{align*}
\left(  \frac{1}{\left\vert Q_{\alpha_{1}}^{1}\right\vert }\int_{Q_{\alpha
_{1}}^{1}}\left\vert f\right\vert ^{p}\right)  ^{1/p}  &  \leq c_{n,p}\left(
\frac{1}{\left\vert Q_{\alpha_{1}}^{1}\right\vert }\int_{Q_{\alpha_{1}}^{1}%
}\left(  f_{\Omega_{n+1},\Omega_{n+2}}^{\#}\right)  ^{p}\right)  ^{1/p}\\
\left(  \frac{1}{\left\vert B_{1}\right\vert }\int_{B_{1}}\left\vert
f\right\vert ^{p}\right)  ^{1/p}  &  \leq c_{n,p}^{\prime}\left(  \frac
{1}{\left\vert B_{2}\right\vert }\int_{B_{2}}\left(  f_{\Omega_{n+1}%
,\Omega_{n+2}}^{\#}\right)  ^{p}\right)  ^{1/p}%
\end{align*}
for some constants $c_{n,p},c_{n,p}^{\prime}$ only depending on $n,p$. Also,
removing the assumption $f_{B_{2}}=0$ we can write%
\[
\left(  \frac{1}{\left\vert B_{1}\right\vert }\int_{B_{1}}\left\vert
f-f_{B_{2}}\right\vert ^{p}\right)  ^{1/p}\leq c_{n,p}^{\prime}\left(
\frac{1}{\left\vert B_{2}\right\vert }\int_{B_{2}}\left(  f_{\Omega
_{n+1},\Omega_{n+2}}^{\#}\right)  ^{p}\right)  ^{1/p}.
\]

\end{corollary}

\begin{proof}
We can write%
\[
\frac{1}{\left\vert Q_{\alpha_{1}}^{1}\right\vert }\int_{Q_{\alpha_{1}}^{1}%
}\left\vert f\right\vert =\frac{1}{\left\vert Q_{\alpha_{1}}^{1}\right\vert
}\int_{Q_{\alpha_{1}}^{1}}\left\vert f-f_{B_{2}}\right\vert \leq c_{n}\frac
{1}{\left\vert B_{2}\right\vert }\int_{B_{2}}\left\vert f-f_{B_{2}}\right\vert
\leq c_{n}f_{\Omega_{n+1},\Omega_{n+2}}^{\#}\left(  x\right)
\]
for every $x\in B_{2}$. Averaging this inequality on $Q_{\alpha_{1}}^{1}$ we
get%
\[
\frac{1}{\left\vert Q_{\alpha_{1}}^{1}\right\vert }\int_{Q_{\alpha_{1}}^{1}%
}\left\vert f\right\vert \leq c_{n}\frac{1}{\left\vert Q_{\alpha_{1}}%
^{1}\right\vert }\int_{Q_{\alpha_{1}}^{1}}f_{\Omega_{n+1},\Omega_{n+2}}%
^{\#}\leq c_{n}\left(  \frac{1}{\left\vert Q_{\alpha_{1}}^{1}\right\vert }%
\int_{Q_{\alpha_{1}}^{1}}\left(  f_{\Omega_{n+1},\Omega_{n+2}}^{\#}\right)
^{p}\right)  ^{1/p}%
\]
so that, by recalling Corollary \ref{coroll 1}%
\[
\left(  \frac{1}{\left\vert Q_{\alpha_{1}}^{1}\right\vert }\int_{Q_{\alpha
_{1}}^{1}}\left\vert f\right\vert ^{p}\right)  ^{1/p}\leq c_{n,p}\left(
\frac{1}{\left\vert Q_{\alpha_{1}}^{1}\right\vert }\int_{Q_{\alpha_{1}}^{1}%
}\left(  f_{\Omega_{n+1},\Omega_{n+2}}^{\#}\right)  ^{p}\right)  ^{1/p},
\]
which also implies the second inequality, by the locally doubling condition
and the comparability of the radii of $B_{1},B_{2}$.
\end{proof}

Although, in the previous Corollary, the second inequality has the pleasant
feature of involving balls instead of dyadic cubes (however, note the two
different balls appearing at the left hand side of the last inequality),
remember that we cannot choose these balls as we like, since they are related
to dyadic cubes.

In concrete applications of this theory, we could use this result to bound
$\left\Vert f\right\Vert _{L^{p}\left(  \Omega_{n}\right)  }$. To this aim,
recall that the domain $\Omega_{n}$ can be covered (up to a zero measure set)
by a finite union of dyadic cubes of the kind $Q_{\alpha_{1}}^{1}$, but
$\Omega_{n}$ is not covered by the union of the corresponding \emph{smaller
}balls $B_{1}$. We then need to improve the previous corollary, replacing the
dyadic cube $Q$ on the left hand side with a \emph{larger }ball:

\begin{corollary}
\label{coroll 3}For any $n$ and every $k$ large enough, the set $\Omega_{n}$
can be covered by a finite union of balls $B_{R}\left(  x_{i}\right)  $ of
radii comparable to $\delta^{k}$ such that for any such ball $B_{R}$ and every
$f$ supported in $B_{R}$ such that $f\in L^{1}\left(  B_{R}\right)  $,
$\int_{B_{R}}f=0$, and $f_{\Omega_{n+2},\Omega_{n+3}}^{\#}\in L_{loc}%
^{p}\left(  \Omega_{n+1}\right)  $ for some $p\in\lbrack1,\infty)$, one has%
\[
\left\Vert f\right\Vert _{L^{p}\left(  B_{R}\right)  }\leq c_{n,p}\left\Vert
f_{\Omega_{n+2},\Omega_{n+3}}^{\#}\right\Vert _{L^{p}\left(  B_{\gamma
R}\right)  }%
\]
with $\gamma>1$ absolute constant.
\end{corollary}

\begin{proof}
Applying Lemma \ref{Covering Lemma}, let us (essentially) cover $\Omega_{n}$
with a finite union of dyadic cubes
\[
Q_{\alpha}^{k}\subset B\left(  z_{\alpha}^{k},c_{1}\delta^{k}\right)  \subset
F_{\alpha}^{k}\subset B\left(  z_{\alpha}^{k},c^{\prime}\delta^{k}\right)
\]
(where the inclusion $B\left(  z_{\alpha}^{k},c_{1}\delta^{k}\right)  \subset
F_{\alpha}^{k}$ is only essential). We claim that the balls $B\left(
z_{\alpha}^{k},c_{1}\delta^{k}\right)  $ are the required covering of
$\Omega_{n}$. To see this, let $f$ be supported in $B\left(  z_{\alpha}%
^{k},c_{1}\delta^{k}\right)  $ and with vanishing integral. Then the same is
true for $f$ with respect to the larger ball $B\left(  z_{\alpha}%
^{k},c^{\prime}\delta^{k}\right)  $. We can then apply Corollary
\ref{coroll 2} to each dyadic cube $Q_{\beta}^{k}$ which constitutes
$F_{\alpha}^{k}$, writing:%
\[
\left(  \frac{1}{\left\vert Q_{\beta}^{k}\right\vert }\int_{Q_{\beta}^{k}%
}\left\vert f\right\vert ^{p}\right)  ^{1/p}\leq c_{n,p}\left(  \frac
{1}{\left\vert Q_{\beta}^{k}\right\vert }\int_{Q_{\beta}^{k}}\left(
f_{\Omega_{n+2},\Omega_{n+3}}^{\#}\right)  ^{p}\right)  ^{1/p}%
\]
(note that the local sharp function is $f_{\Omega_{n+2},\Omega_{n+3}}^{\#}$
because we are using dyadic balls related to $\Omega_{n+1}$), that is%
\[
\int_{Q_{\beta}^{k}}\left\vert f\right\vert ^{p}\leq c_{n,p}^{p}\int
_{Q_{\beta}^{k}}\left(  f_{\Omega_{n+2},\Omega_{n+3}}^{\#}\right)  ^{p}.
\]
Adding these inequalities for all the cubes $Q_{\beta}^{k}$ in $F_{\alpha}%
^{k}$ we get%
\begin{align*}
&  \left(  \int_{B\left(  z_{\alpha}^{k},c_{1}\delta^{k}\right)  }\left\vert
f\right\vert ^{p}\right)  ^{1/p}\leq\left(  \int_{F_{\alpha}^{k}}\left\vert
f\right\vert ^{p}\right)  ^{1/p}\\
&  \leq c_{n,p}\left(  \int_{F_{\alpha}^{k}}\left(  f_{\Omega_{n+2}%
,\Omega_{n+3}}^{\#}\right)  ^{p}\right)  ^{1/p}\leq c_{n,p}\left(
\int_{B\left(  z_{\alpha}^{k},c^{\prime}\delta^{k}\right)  }\left(
f_{\Omega_{n+2},\Omega_{n+3}}^{\#}\right)  ^{p}\right)  ^{1/p}%
\end{align*}
which is our assertion, with $R=c_{1}\delta^{k},\gamma=c^{\prime}/c_{1}$.
\end{proof}

\section{Local $BMO$ and John-Nirenberg inequality}

We start defining the space of functions with locally bounded mean oscillation
in a locally homogeneous space:

\begin{definition}
Let $f\in L^{1}(\Omega_{n+1})$. We say that $f$ belongs to $BMO(\Omega
_{n},\Omega_{n+1})$ if
\[
\lbrack f]_{n}\equiv\sup_{x\in\Omega_{n},r\leq2\epsilon_{n}}\frac{1}%
{\mu(B(x,r)}\int_{B(x,r)}|f(y)-f_{B(x,r)}|d\mu(y)<\infty.
\]

\end{definition}

The main result in this section is the following.

\begin{theorem}
[Local John-Nirenberg inequality]\label{john-nirenberg}There exist positive
constants $b_{n},R_{n}$ such that $\forall f\in BMO(\Omega_{n},\Omega_{n+1})$
and for any ball $B(a,R)$, with $a\in\Omega_{n}$ and $R\leq R_{n}$, the
following inequality holds true
\begin{equation}
\mu\left(  \{x\in B(a,R):\left\vert f(x)-f_{B(a,R)}\right\vert >\lambda
\}\right)  \leq2e^{\frac{-b_{n}\lambda}{[f]_{n}}}\mu(B(a,R))\,\quad
\forall\lambda>0.\label{JN ineq}%
\end{equation}

\end{theorem}

\begin{remark}
\label{Remark R_n}As will appear from the proof, the constant $R_{n}$ is
strictly smaller than the number $2\varepsilon_{n}$ appearing in the
definition of $BMO(\Omega_{n},\Omega_{n+1})$. Explicitly, we will see that%
\begin{equation}
R_{n}=\frac{2\varepsilon_{n}}{B_{n+1}\left(  \frac{9}{2}B_{n+1}^{2}%
+3B_{n+1}+1\right)  }.\label{R_n}%
\end{equation}

\end{remark}

\begin{proof}
We can assume $[f]_{n}=1$, since (\ref{JN ineq}) does not change dividing both
$f$ and $\lambda$ for a constant.

Let%
\[
K_{n}=2B_{n+1}+3B_{n+1}^{2}%
\]
be the constant appearing in Vitali covering Lemma \ref{Lemma Vitali},%
\begin{align*}
\alpha_{n} &  =B_{n+1}\left(  \frac{3}{2}K_{n}+1\right)  \\
R_{n} &  =\frac{2\varepsilon_{n}}{\alpha_{n}}.
\end{align*}

Let $a\in\Omega_{n},$ $R\leq R_{n}$ and let $S=B(a,R)$ (since $R<\varepsilon
_{n}$, $S\subseteq\Omega_{n+1}$). The proof consists in an iterative construction.

\emph{Step 1. }We will prove that there exists a family of balls $\left\{
S_{j}\right\}  _{j=1}^{\infty}\subset S$ and constants $c,\lambda_{0}\geq1$
depending on $n$ such that:

i)
\[
\{x\in S:\left\vert f(x)-f_{S}\right\vert >\lambda_{0}\}\subset\bigcup
\limits_{j=1}^{\infty}S_{j}\subset S;
\]

ii)%
\[
\sum_{j=1}^{\infty}\mu\left(  S_{j}\right)  \leq\frac{1}{2}\mu\left(
S\right)  ;
\]

iii)%
\[
\left\vert f_{S}-f_{S_{j}}\right\vert \leq c\lambda_{0}%
\]

To prove this we start defining the maximal operator associated to $S$
letting, for any $x\in S,$%
\[
M_{S}f(x)=\sup\left\{  \frac{1}{\mu(B)}\int_{B}|f(y)-f_{S}|d\mu
(y):B\,\,\text{ball},x\in B,B\subseteq\alpha_{n}S\right\}
\]
where $\alpha_{n}S=B(a,\alpha_{n}R)\subseteq\Omega_{n+1}$ since $\alpha
_{n}R\leq2\varepsilon_{n}$.

We claim that there exists $A=A(n)>0$ such that for all $t>0$
\begin{equation}
\mu\left(  \{x\in S:M_{S}f(x)>t\}\right)  \leq\frac{A}{t}\mu(S)\,.\label{4}%
\end{equation}

To show this, let $t>0$ and let
\[
U_{t}=\{x\in S:M_{S}f(x)>t\}.
\]
For every $x\in U_{t}$ there exists a ball $B_{x}$ such that $x\in
B_{x}\subseteq\alpha_{n}S$ and
\[
\mu(B_{x})<\frac{1}{t}\int_{B_{x}}|f(y)-f_{S}|d\mu\,.
\]
Now, by Vitali Lemma \ref{Covering Lemma} there exists a countable
subcollection of disjoint balls $\{B(x_{i},r_{i})\}$ such that
\[
U_{t}\subseteq\bigcup\limits_{i=1}^{\infty}B(x_{i},K_{n}r_{i})\,.
\]
Then, since by definition of $S$ and $M_{S}f$, $\cup_{i=1}^{\infty}%
B(x_{i},r_{i})\subseteq\alpha_{n}S\subset\Omega_{n+1},$ for some constant
$A=A\left(  n\right)  $ which can vary from line to line we have%
\begin{align*}
\mu(U_{t}) &  \leq\mu\left(  \bigcup\limits_{i=1}^{\infty}B(x_{i},K_{n}%
r_{i})\right)  \leq\sum_{i=1}^{\infty}\mu\left(  B(x_{i},K_{n}r_{i})\right)
\leq A\sum_{i=1}^{\infty}\mu(B(x_{i},r_{i}))\\
&  \leq\sum_{i=1}^{\infty}\frac{A}{t}\int_{B(x_{i},r_{i})}\left\vert
f-f_{S}\right\vert d\mu=\frac{A}{t}\int_{\cup_{i=1}^{\infty}B(x_{i},r_{i}%
)}\left\vert f-f_{S}\right\vert d\mu\\
&  \leq\frac{A}{t}\int_{\alpha_{n}S}\left\vert f-f_{S}\right\vert d\mu
\leq\frac{A}{t}\left\{  \int_{\alpha_{n}S}\left\vert f-f_{\alpha_{n}%
S}\right\vert d\mu+\int_{\alpha_{n}S}\left\vert f_{\alpha_{n}S}-f_{S}%
\right\vert d\mu\right\}  \\
&  \leq\frac{A}{t}\left\{  \mu(\alpha_{n}S)[f]_{n}+\mu(\alpha_{n}%
S)|f_{S}-f_{\alpha_{n}S}|\right\}  \\
&  \leq\frac{A}{t}\mu(\alpha_{n}S)[f]_{n}\\
&  \leq\frac{A}{t}\mu(S)
\end{align*}
where we exploited the assumption $[f]_{n}=1$. Hence (\ref{4}) is proved.

Let now $\lambda_{0}>A$, we consider the following open set%
\[
U=\{x\in S:M_{S}f(x)>\lambda_{0}\}\,.
\]
We have, by (\ref{4}),%
\begin{equation}
\mu(U\cap S)=\mu(U)\leq\frac{A}{\lambda_{0}}\mu(S)<\mu(S)\,,\label{1}%
\end{equation}
from which
\[
S\cap U^{c}\not =\emptyset\,.
\]
Then for any $x\in S$ we set%
\[
r(x)=\frac{1}{2K_{n}}\rho(x,U^{c})\quad\forall x\in U.
\]
If $x,y\in S$ we have $\rho(x,y)\leq2B_{n+1}R\,.$ Then $\forall x\in U$
(taking a point $y\in U^{c}\cap S$ in the following inequality)%
\[
r(x)\leq\frac{1}{2K_{n}}\rho(x,y)\leq\frac{1}{2B_{n+1}\left(  2+3B_{n+1}%
\right)  }2B_{n+1}R\leq\frac{R}{5}%
\]
If $y\in B(x,K_{n}r(x))$ for some $x\in U$, we have
\[
\rho(y,x)<K_{n}r(x)=\frac{K_{n}}{2K_{n}}\rho(x,U^{c})<\rho(x,U^{c})
\]
then $y\in U$, from which
\begin{equation}
B(x,K_{n}r(x))\subseteq U\,.\label{JN inclusion 1b}%
\end{equation}
On the other hand%
\begin{equation}
U\subset\bigcup\limits_{x\in U}B\left(  x,r\left(  x\right)  \right)
\label{JN inclusion 1}%
\end{equation}
and from the Vitali Lemma there exists a countable sequence of disjoint balls
$\{B(x_{j},r_{j})\}$ ($r_{j}=r(x_{j})$) such that%
\[
U\subset\bigcup\limits_{j=1}^{\infty}B\left(  x_{j},K_{n}r_{j}\right)
\]
which by the inclusion (\ref{JN inclusion 1b}) means that%
\begin{equation}
U=\bigcup\limits_{j=1}^{\infty}B\left(  x_{j},K_{n}r_{j}\right)  .\label{JN 3}%
\end{equation}

Moreover $B(x_{j},3K_{n}r_{j})\cap U^{c}\not =\emptyset$ $\forall
j\in\mathbb{N}$ and $B(x_{j},3K_{n}r_{j})\subseteq\alpha_{n}S$ since
$\alpha_{n}=B_{n+1}\left(  \frac{3K_{n}}{2}+1\right)  $.

If $y\in B(x_{j},3K_{n}r_{j})\cap U^{c}$, then $M_{s}f(y)\leq\lambda_{0}$ and%
\begin{equation}
\frac{1}{\mu(B(x_{j},3K_{n}r_{j}))}\int_{B(x_{j},3K_{n}r_{j})}\left\vert
f-f_{S}\right\vert d\mu\leq\lambda_{0}.\label{JN 2}%
\end{equation}

We now set%
\[
S_{j}=B(x_{j},K_{n}r_{j})
\]
and we have, by (\ref{JN 2})%
\begin{align}
\left\vert f_{S}-f_{S_{j}}\right\vert  &  =\left\vert \frac{1}{\mu(S_{j})}%
\int_{S_{j}}fd\mu-f_{S}\right\vert \leq\frac{1}{\mu(S_{j})}\int_{S_{j}%
}\left\vert f-f_{S}\right\vert d\mu\label{14}\\
&  \leq\frac{c}{\mu(B(x_{j},3K_{n}r_{j}))}\int_{B(x_{j},3K_{n}r_{j}%
)}\left\vert f-f_{S}\right\vert d\mu\leq c\lambda_{0}\nonumber
\end{align}
which is point iii).

By the differentiation theorem we have that for a.e. $x\in S\setminus\cup
_{j}S_{j}$ (that by (\ref{JN 3}) implies that $x\in U^{c}$ so that
$M_{S}f(x)\leq\lambda_{0}$)
\[
\left\vert f(x)-f_{S}\right\vert \leq\lambda_{0}.
\]
This means that%
\[
\{x\in S:\left\vert f(x)-f_{S}\right\vert >\lambda_{0}\}\subset\bigcup
\limits_{j=1}^{\infty}S_{j}\subset S
\]
which is point i).

Moreover, by the doubling property (H7) and (\ref{1}),%
\begin{align*}
\sum_{j=1}^{\infty}\mu(S_{j}) &  \leq c\sum_{j=1}^{\infty}\mu(B(x_{j}%
,r_{j}))=c\mu\left(  \bigcup\limits_{j=1}^{\infty}B(x_{j},r_{j})\right)
\leq\\
&  \leq c\mu(U)\leq c\frac{A}{\lambda_{0}}\mu(S)=\frac{1}{2}\mu\left(
S\right)  ,
\end{align*}
having finally chosen $\lambda_{0}=2cA,$ so that also point ii) is proved and
step 1 is completed.

\emph{Step 2 }consists in doing the same construction on each ball $S_{j}$
constructed in Step 1, which allows to conclude that, for every $j_{1}%
=1,2,...$ there exists a sequence of balls $\left\{  S_{j_{1}j_{2}}\right\}
_{j_{2}=1}^{\infty}\subset S_{j_{1}}$ such that (for the same constants
$c,\lambda_{0}$ of Step 1)

i)
\[
\{x\in S_{j_{1}}:\left\vert f(x)-f_{S_{j_{1}}}\right\vert >\lambda
_{0}\}\subset\bigcup\limits_{j_{2}=1}^{\infty}S_{j_{1}j_{2}}\subset S_{j_{1}};
\]

ii)%
\[
\sum_{j_{2}=1}^{\infty}\mu\left(  S_{j_{1}j_{2}}\right)  \leq\frac{1}{2}%
\mu\left(  S_{j_{1}}\right)  ;
\]

iii)%
\[
\left\vert f_{S_{j_{1}}}-f_{S_{j_{1}j_{2}}}\right\vert \leq c\lambda_{0}.
\]
Point ii) of Step 2 and Step 1 imply%
\[
\sum_{j_{1},j_{2}=1}^{\infty}\mu\left(  S_{j_{1}j_{2}}\right)  \leq\frac{1}%
{2}\sum_{j_{1}=1}^{\infty}\mu\left(  S_{j_{1}}\right)  \leq\frac{1}{4}%
\mu\left(  S\right)  .
\]
Also, point i) of Step 2 and point iii) of Step 1, imply that for a.e. $x\in
S_{j_{1}}\setminus\bigcup\limits_{j_{2}=1}^{\infty}S_{j_{1}j_{2}}$,%
\[
\left\vert f(x)-f_{S}\right\vert \leq\left\vert f(x)-f_{S_{j_{1}}}\right\vert
+\left\vert f_{S_{j_{1}}}-f_{S}\right\vert \leq\lambda_{0}+c\lambda
_{0}<2c\lambda_{0}\,
\]
which means that
\begin{equation}
\{x\in S_{j_{1}}:\left\vert f(x)-f_{S}\right\vert >2c\lambda_{0}%
\}\subset\bigcup\limits_{j_{2}=1}^{\infty}S_{j_{1}j_{2}}\subset S_{j_{1}}.
\label{JN 4}%
\end{equation}
However, point i) of Step 1 implies%
\[
\{x\in S:\left\vert f(x)-f_{S}\right\vert >2c\lambda_{0}\}\subset\{x\in
S:\left\vert f(x)-f_{S}\right\vert >\lambda_{0}\}\subset\bigcup\limits_{j_{1}%
=1}^{\infty}S_{j_{1}}%
\]
hence (\ref{JN 4}) rewrites as%
\[
\{x\in S:\left\vert f(x)-f_{S}\right\vert >2c\lambda_{0}\}\subset
\bigcup\limits_{j_{1},j_{2}=1}^{\infty}S_{j_{1}j_{2}}%
\]
and, letting $\lambda_{1}=c\lambda_{0},$%
\begin{equation}
\mu\left(  \{x\in S:\left\vert f(x)-f_{S}\right\vert >2\lambda_{1}\}\right)
\leq\frac{1}{4}\mu\left(  S\right)  . \label{JN 6}%
\end{equation}
Relation (\ref{JN 6}) summarizes the joint consequences of Steps 1 and 2.

Proceeding this way the iterative construction, at \emph{Step }$N$\emph{ }we
will have that:%
\begin{equation}
\mu\left(  \{x\in S:\left\vert f(x)-f_{S}\right\vert >N\lambda_{1}\}\right)
\leq\frac{1}{2^{N}}\mu\left(  S\right)  . \label{JN 7}%
\end{equation}

Now, let $\lambda>0$. If $\lambda\geq\lambda_{1}$, let $N$ be the positive
integer such that
\[
N\lambda_{1}<\lambda\leq(N+1)\lambda_{1},
\]
then
\begin{align*}
\mu\left(  \left\{  x\in S:\left\vert f(x)-f_{S}\right\vert >\lambda\right\}
\right)   &  \leq\mu\left(  \left\{  x\in S:\left\vert f(x)-f_{S}\right\vert
>N\lambda_{1}\right\}  \right) \\
&  \leq\frac{1}{2^{N}}\mu\left(  S\right)  =e^{-N\log2}\mu(S)\leq2e^{-\left(
\frac{\log2}{\lambda_{1}}\right)  \lambda}\mu(S)
\end{align*}

Finally, if $0<\lambda\leq\lambda_{1},$
\[
\mu\left(  \left\{  x\in S:\left\vert f(x)-f_{S}\right\vert >\lambda\right\}
\right)  \leq\mu(S)\leq2e^{-\left(  \frac{\log2}{\lambda_{1}}\right)  \lambda
}\mu(S)
\]
and the assertion follows (recall we are assuming $[f]_{n}=1$), with
$b_{n}=\frac{\log2}{\lambda_{1}}.$
\end{proof}

\begin{definition}
Let $p\in\left(  1,+\infty\right)  $. We say that $f$ belongs to
$BMO^{p}(\Omega_{n},\Omega_{n+1})$ if $f\in L^{p}(\Omega_{n+1})$ and%
\[
\lbrack f]_{p,n}\equiv\sup_{x\in\Omega_{n},r\leq R_{n}}\left(  \frac{1}%
{\mu(B(x,r))}\int_{B(x,r)}|f(y)-f_{B(x,r)}|^{p}d\mu(y)\right)  ^{1/p}<\infty\,
\]
where $R_{n}$ is the constant appearing in (\ref{R_n}), strictly smaller than
$2\varepsilon_{n}$.
\end{definition}

Now we compare the spaces $BMO(\Omega_{n},\Omega_{n+1})$ and $BMO^{p}%
(\Omega_{n},\Omega_{n+1})$.

\begin{theorem}
\label{Thm BMOp}For any $p\in\left(  1,\infty\right)  $ and $n$ we have
\[
BMO^{p}(\Omega_{n},\Omega_{n+1})=BMO(\Omega_{n},\Omega_{n+1}).
\]
Moreover, there exists a positive constant $c_{n,p}$ such that for any $f\in
BMO(\Omega_{n},\Omega_{n+1}),$
\begin{equation}
\lbrack f]_{p,n}\leq c_{n,p}[f]_{n}\,.\label{corollario_finale}%
\end{equation}
In particular, $BMO(\Omega_{n},\Omega_{n+1})\subseteq L^{p}\left(  \Omega
_{n}\right)  $ for every $p\in\left(  1,\infty\right)  $.
\end{theorem}

\begin{remark}
Comparing this result with those about the local Fefferman-Stein function
proved in the previous section (for instance, Corollary \ref{coroll 3}), we
see that the present theorem is a \textquotedblleft local\textquotedblright%
\ result in a different sense. Here, in the upper bound
(\ref{corollario_finale}), there is not an enlargement of the domain, passing
from the left to the right hand side; instead, the local seminorms $[f]_{p,n}$
are computed taking the supremum over balls of radii $r\leq R_{n}$, which is a
stricter condition than the bound $r\leq2\varepsilon_{n}$ defining the
seminorm $[f]_{n}$.
\end{remark}

\begin{proof}
Let $f\in BMO^{p}(\Omega_{n},\Omega_{n+1})$. By H\"{o}lder's inequality we can
write, for every $x\in\Omega_{n},r\leq R_{n}$,%
\begin{align*}
& \frac{1}{\mu(B(x,r))}\int_{B(x,r)}|f(y)-f_{B(x,r)}|d\mu(y)\\
& \leq\left(  \frac{1}{\mu(B(x,r))}\int_{B(x,r)}|f(y)-f_{B(x,r)}|^{p}%
d\mu(y)\right)  ^{1/p}\leq\lbrack f]_{p,n}<\infty.
\end{align*}
On the other hand, if $R_{n}<r<2\varepsilon_{n}$ we have
\begin{align*}
\frac{1}{\mu(B(x,r))}\int_{B(x,r)}|f(y)-f_{B(x,r)}|d\mu(y)  & \leq2\frac
{1}{\mu(B(x,r))}\int_{B(x,r)}|f(y)|d\mu(y)\\
& \leq2\frac{1}{\mu(B(x,R_{n}))}\int_{\Omega_{n+1}}|f(y)|d\mu(y)\\
& \leq c_{n}\left\Vert f\right\Vert _{L^{p}\left(  \Omega_{n+1}\right)  }%
\end{align*}
because%
\[
\inf_{x\in\Omega_{n}}\mu(B(x,R_{n}))\geq c_{n}>0,
\]
as can be easily proved as a consequence of the local doubling condition.
Therefore $[f]_{n}\,<\infty$ and $f\in BMO(\Omega_{n},\Omega_{n+1})$.

Conversely, to prove (\ref{corollario_finale}), let $B$ be a ball centered in
$x\in\Omega_{n}$ with radius $r\leq R_{n}$. Then by Theorem
\ref{john-nirenberg} we have%
\begin{align*}
\int_{B}|f(y)-f_{B}|^{p}d\mu(y) &  =\int_{0}^{+\infty}p\lambda^{p-1}\mu\left(
\left\{  z\in B:\left\vert f(z)-f_{B}\right\vert >\lambda\right\}  \right)
d\lambda\\
&  \leq2\int_{0}^{+\infty}p\lambda^{p-1}e^{-b_{n}\lambda/[f]_{n}}%
\mu(B)d\lambda\,\\
&  =\mu(B)[f]_{n}^{p}2p\int_{0}^{+\infty}t^{p-1}e^{-b_{n}t}dt,
\end{align*}
from which
\[
\left(  \frac{1}{\mu(B)}\int_{B}\left\vert f(y)-f_{B}\right\vert ^{p}%
d\mu(y)\right)  ^{1/p}\leq\left(  2p\int_{0}^{+\infty}t^{p-1}e^{-b_{n}%
t}dt\right)  ^{1/p}[f]_{n}=c_{n,p}[f]_{n}%
\]
which gives (\ref{corollario_finale}) and the inclusion $BMO(\Omega_{n}%
,\Omega_{n+1})\subseteq BMO^{p}(\Omega_{n},\Omega_{n+1})$.

Finally, to show that $f\in L^{p}\left(  \Omega_{n}\right)  $ we can cover
$\Omega_{n}$ with a finite collection of balls $B\left(  x,R_{n}\right)  $
with $x\in\Omega_{n}$, writing%
\begin{align*}
\left(  \frac{1}{\mu(B)}\int_{B}\left\vert f(y)\right\vert ^{p}d\mu(y)\right)
^{1/p}  & \leq\left(  \frac{1}{\mu(B)}\int_{B}\left\vert f(y)-f_{B}\right\vert
^{p}d\mu(y)\right)  ^{1/p}+\left\vert f_{B}\right\vert \\
& \leq[f]_{p,n}+\frac{1}{\mu(B)}\left\Vert f\right\Vert _{L^{1}\left(
\Omega_{n+1}\right)  }<\infty
\end{align*}
which implies the finiteness of $\left\Vert f\right\Vert _{L^{p}\left(
\Omega_{n}\right)  }$.
\end{proof}

\bigskip

\textsc{Marco Bramanti}

\textsc{Dipartimento di Matematica}

\textsc{Politecnico di Milano}

\textsc{Via Bonardi 9}

\textsc{20133 Milano, ITALY}

\texttt{marco.bramanti@polimi.it}

\bigskip

\textsc{Maria Stella Fanciullo}

\textsc{Dipartimento di Matematica e Informatica}

\textsc{Universit\`{a} di Catania}

\textsc{Viale Andrea Doria 6}

\textsc{95125 Catania, ITALY}

\texttt{fanciullo@dmi.unict.it}

\end{document}